\documentclass [11pt]{amsart}
\usepackage{amsmath}
\usepackage{amssymb}
\usepackage{comment}

\newtheorem{thm}{Theorem}[section]

\newtheorem{lem}[thm]{Lemma}

\theoremstyle{definition}
\newtheorem{defn}[thm]{Definition}
\theoremstyle{remark}

\numberwithin{equation}{section}

\newcommand{\beas}{\begin{eqnarray*}}
\newcommand{\eeas}{\end{eqnarray*}}
\newcommand{\bes} {\begin{equation*}}
\newcommand{\ees} {\end{equation*}}
\newcommand{\be} {\begin{equation}}
\newcommand{\ee} {\end{equation}}
\newcommand{\bea} {\begin{eqnarray}}
\newcommand{\eea} {\end{eqnarray}}
\newcommand{\ra} {\rightarrow}
\newcommand{\txt} {\textmd}
\newcommand{\ds} {\displaystyle}
\newcommand{\R}{\mathbb R}
\newcommand{\C}{\mathbb C}

\newcommand{\N}{\mathbb N}
\newcommand{\Z}{\mathbb Z}

\newcommand{\f}{\frac}

\newcommand{\al}{\alpha}
\newcommand{\pl}{\partial}
\newcommand{\lt}{\left}
\newcommand{\rt}{\right}

\newcommand{\g}{\mathfrak{g}}
\newcommand{\z}{\mathfrak{z}}
\newcommand{\vv}{\mathfrak{v}}
\newcommand{\h}{\mathfrak{h}}
\newcommand{\la}{\lambda}

\begin{document}

\title[Roe-Strichartz Theorem on two step nilpotent Lie groups] {Roe-Strichartz Theorem on two step nilpotent Lie groups}

\author[Sayan Bagchi]{Sayan Bagchi}
\address[Sayan Bagchi]{Department of Mathematics and Statistics, Indian Institute of Science Education and Research, Campus Road, Mohanpur, West Bengal 741246, India} \email{sayansamrat@gmail.com}

\author[Ashisha Kumar]{Ashisha Kumar}
\address[Ashisha Kumar]{Discipline of Mathematics, Indian Institute of Technology Indore, Khandwa Road, Simrol, Indore, 453552 India} \email{akumar@iiti.ac.in}

\author[Suparna Sen]{Suparna Sen}
\address[Suparna Sen]{Department of Pure Mathematics, University of Calcutta, 35 Ballygunge Circular Road, Kolkata 700019, India} \email{suparna29@gmail.com}


\begin{abstract}
Strichartz characterized eigenfunctions of the Laplacian on Euclidean spaces by boundedness conditions which generalized a result of Roe for the one-dimensional case. He also proved an analogous statement for the sublaplacian on the Heisenberg groups. In this paper, we extend this result to connected, simply connected two step nilpotent Lie groups.
\end{abstract}



\maketitle

\section{Introduction}
J. Roe proved a result (see \cite{Ro}) that if a function on the real line has the property that all its derivatives and antiderivatives are bounded by a uniform bound then it is a linear combination of sine and cosine functions. Strichartz extended Roe's result to $\mathbb{R}^n$ in \cite{S} as follows:
\begin{thm}
Let $\{f_k\}_{k\in \Z}$ be a doubly infinite sequence of
functions on $\R^n$ satisfying \beas
\Delta f_k &=& f_{k+1},\\
\|f_k\|_{\infty} &\leq& M, \quad \text{for all}\ k \in \Z, \eeas
where $\Delta$ is the Laplacian on $\R^n$ i.e. $\Delta =
\dfrac{\partial^2}{\partial x_1^2} + \cdots +
\dfrac{\partial^2}{\partial x_n^2}.$ Then $\Delta f_0=-f_{0}.$
\end{thm}
\noindent For some related results and different proofs of Roe-Strichartz Theorem on Euclidean spaces, we refer the reader to \cite{B, H, HR}. In \cite{S}, Strichartz showed that this result fails for hyperbolic $3$-space. In fact, this result is not true for any Riemannian symmetric space of noncompact type. However, some interesting modified versions of this result are proved in \cite{KRS, RS}.

Strichartz has proved this result on Heisenberg group corresponding to the sublaplacian on the Heisenberg group in \cite{S}. We wish to prove this result on connected, simply connected two step nilpotent Lie groups. To state our main theorem, some definitions are in order.

Let $G$ be a connected, simply connected two step nilpotent Lie group, $\g$ be the Lie algebra of $G$ and $\z$ denote the center of $\g$. We consider a subspace $\vv$ of $\g$ complementary to $\z$ so that $\g$ has the decomposition $\g = \vv \oplus \z$. We choose an inner product $\langle \cdot,\cdot \rangle$ on $\g$ so that the above decomposition is orthogonal. Since $G$ is nilpotent the exponential map from $\mathfrak g$ to $G$ is an analytic diffeomorphism. Thus the elements of $G$ can be identified with the elements of its Lie algebra $\g$ via the exponential map. We shall denote the element $\exp(V+T) \in G$ by $(V,T)$ where $V\in \vv$ and $T\in \z$. The product law on $G$ is given by the Baker-Campbell-Hausdorff formula
\bes
(V,T)(V',T')=\left(V+V',T+T'+\frac{1}{2}{[V,V']}\right), \: \:\:\: \txt{ for all } V,V'\in \vv, \:\: T,T'\in \z.
\ees
Let $\{V_1, V_2, \cdots, V_m\}$ be an orthonormal basis of $\vv.$ We define the sublaplacian $\mathcal{L}$ of $G$ by
\be \label{sublaplacian}
\mathcal{L} = V_1^2 + V_2^2 + \cdots + V_m^2
\ee
We will now state Roe-Strichartz Theorem for a connected, simply connected two step nilpotent Lie group $G$.

\begin{thm} \label{theorem}
Let $\{f_k\}_{k\in \Z}$ be a doubly infinite sequence of functions on a connected, simply connected two step nilpotent Lie group $G.$ Suppose there exists a constant $M>0$ such that
\bea
\label{seq} \mathcal{L}f_k &=& f_{k+1}, \quad \text{ for all}\ k \in \Z,\\
\label{bdd} \txt{ and } \quad \|f_k\|_{\infty} &\leq& M, \quad \text{ for all}\ k \in \Z.
\eea
Then $\mathcal{L}f_0=-f_{0}$.
\end{thm}
The main idea of our proof of Theorem \ref{theorem} is inspired from Strichartz's proof of the result on the Heisenberg group. However, the structure of general two step nilpotent Lie groups is much more complicated than the Heisenberg group. As a result, our proof is not at all a straight forward generalisation of Strichartz's proof. Firstly, we generalise the result on two step MW groups (to be defined subsequently). Then we use the idea in \cite{MR} to embed a two step non-MW group inside a bigger two step MW group to prove the result on two step non-MW groups. Unlike the Heisenberg group, the structure of two step MW groups is much more complicated in the sense that the symplectic basis depends on the elements of the dual. So we need to make significant changes in order to generalise the original idea of proof. Moreover, we need to invoke Wiener Tauberian theorem for two step nilpotent Lie groups in \cite{L} to prove the result.

Our paper is organised as follows: In the subsequent section we will discuss some preliminaries for two step nilpotent Lie groups. In the third section we will prove Theorem \ref{theorem} for two-step MW groups. The last section is devoted for proving the above theorem for two-step non-MW groups. Throughout the paper $C$ denotes a constant which may vary and $e_j$ denotes the unit vector in $\R^n$ whose $j$th coordinate is $1$ for $j = 1, \cdots, n.$

\section{Preliminaries on Two Step Nilpotent Lie Groups}

Representation theory of connected, simply connected two step nilpotent Lie groups and the Plancherel theorem is described in \cite{R}. In this section, we shall briefly describe the basic facts needed to make this paper self contained. For representation theory of general connected, simply connected nilpotent Lie groups, please see \cite{CG}.

Let $G$ be a connected, simply connected two step nilpotent Lie group and $\g,$ $\z,$ $\vv$ be as defined in the introduction. Let $\g^*, \vv^*$ and $\z^*$ be the dual vector spaces of $\g, \vv$ and $\z$ respectively. For each $\la \in \z^*$ we consider the alternating bilinear form $B_\la$ on $\vv$ defined by
\bes
B_\la(V,V')=\la([V,V']), \quad \txt{ for all }V, V'\in \vv.
\ees
\begin{defn}
Let $G$ be a connected, simply connected two step nilpotent Lie group. If the bilinear form $B_\la$ is nondegenerate for some $\la \in \z^*$ then we call $G$ a two step nilpotent Lie group with MW condition or two step MW group. Otherwise, if $B_\la$ is degenerate for all $\la \in \z^*,$ then $G$ is called a two step nilpotent Lie group without MW condition or two step non-MW group.
\end{defn}

We will now describe the representation theory of two step MW groups. Since we will prove the Roe-Strichartz theorem for non-MW groups by reducing it to the MW case, we refrain from describing the representation theory of two step non-MW groups. It is known that for a two step MW group the set defined by
\bes
\Lambda_0 = \{\la \in \z^* \mid B_{\la} \text{ is nondegenerate}\}
\ees
is a Zariski open subset of $\z^*$ and it gives a parametrization of the irreducible unitary representations of $G$ relevant to the Plancherel measure. Since $B_\la$ is nondegenerate, for each $\la \in \Lambda_0$ we have $\dim \vv = m = 2n$ for some $n \in \N$. Now, identifying $\vv$ with $\C^n,$ we consider an orthonormal basis $\{Z_1, Z_2, \cdots, Z_n\}$ of $\vv.$ From the properties of an alternating bilinear form, it follows that for a fixed $\la \in \Lambda_0,$ the orthonormal basis $\{Z_1, Z_2, \cdots, Z_n\}$ of $\vv$ can be transformed to another orthonormal basis $\{Z_1(\la), Z_2(\la), \cdots, Z_n(\la)\}$ of $\vv$ by a transformation $D_{\la},$ where $Z_j(\la) = X_j(\la) + i Y_j(\la),$ for $j = 1, 2, \cdots, n$ such that there exist positive numbers $d_j(\la)>0$ satisfying
\bes
\la([X_i(\la),Y_j(\la)])=\delta_{i,j}d_j(\la), \quad 1\leq i,j \leq n.
\ees
It is known that (see \cite{CRS}, \cite{MR}) one can find a smaller Zariski-open subset $\Lambda \subset \Lambda_0$ of $\z^*$ such that $d_j$ is a smooth function on $\Lambda.$ We note that $\Lambda$ is a set of full measure in $\R^k$, where $k$ is the dimension of $\z$. Further, we can also choose $X_j(\lambda)$ and $Y_j(\lambda)$ to depend smoothly on $\lambda.$ So it follows that the matrix coefficients of the transformation $D_\la$ depend smoothly on $\la.$

Let us fix $\la \in \Lambda.$ We define the subspaces of $\vv$ given by
\beas
\xi_\la &=& \txt{span}_{\R}\{X_1(\la),\cdots,X_n(\la)\}, \\
\eta_\la &=& \txt{span}_\R\{Y_1(\la),\cdots,Y_n(\la)\}.
\eeas
Then we have the decomposition
\bes
\g = \vv \oplus \z = \xi_\la \oplus \eta_\la \oplus \mathfrak z.
\ees
With respect to the above decomposition, we shall denote any element $(z,t) \in \g$ by $(x(\la),y(\la),t)$. Let $\{T_1, \cdots, T_k\}$ be an orthonormal basis of $\z$. We shall write
\bes
(x(\la),y(\la),t)=\sum_{j=1}^{n}{x_j(\la)X_j(\la)}+\sum_{j=1}^{n}{y_j(\la)Y_j(\la)}+\sum_{j=1}^{k}{t_jT_j}.
\ees
The basis $\{X_1(\la), Y_1(\la), \cdots, X_n(\la), Y_n(\la), T_1, \cdots, T_k \}$ of $\g$ is called an almost symplectic basis.

Let $\{ T_1^*, \cdots, T_k^* \}$ denote the dual basis in $\z^*$. Using the almost symplectic basis we now describe an irreducible unitary representation $\pi_{\la}$ of $G$ realized on $L^2(\eta_\la)$ by the following action
\bes
\left(\pi_\la(x(\la),y(\la),t)\phi\right)(\xi(\la))= e^{i\sum_{j=1}^{k}{\la_j t_j} + i\sum_{j=1}^{n}{d_j(\la)\left(x_j(\la)\xi_j(\la)+\frac{1}{2}x_j(\la)y_j(\la)\right)}} \phi(\xi(\la)+y(\la)),
\ees
where $\phi \in L^2(\eta_\la)$, $\la = \la_1 T_1^* + \cdots + \la_k T_k^*$ and $\xi(\la) = \sum_{j=1}^n \xi_j(\la) Y_j(\la)$. We define the Fourier transform of $f\in L^1(G)$ by the operator valued integral
\bes
\widehat{f}(\la)=\int_{\mathfrak z} \int_{\eta_\la} \int_{\xi_\la}{f(x(\la),y(\la),t) ~ \pi_{\la}(x(\la),y(\la),t)~dx(\la)~dy(\la)~dt}, \quad \txt{ for } \lambda \in \Lambda.
\ees
For $\la\in \Lambda$ and $f \in L^1(G)$ we consider the Euclidean Fourier transform of $f$ in the central variable given by
\be \label{centralft}
f^\la(z)=\int_{\mathfrak z} f(z,t) ~ e^{- i\la(t)}~dt, \quad \txt{ for } z \in \vv,
\ee
For $f, g \in L^1(\vv)$ the $\la$-twisted convolution of $f$ and $g$ is defined by
\be \label{twistedconv}
f*_{\la}g(z)=\int_{\vv}~f(z-w)~g(w)~e^{-{\frac{i}{2}}\la([z,w])}~dw,  \quad \txt{ for } z \in \vv.
\ee
It follows from (\ref{centralft}) and (\ref{twistedconv}) that for $f,g \in L^1(G)$ we have
\be \label{twistft}
(f*g)^{\la}(z)=f^{\la}*_{\la}g^{\la}(z),  \quad \txt{ for } z \in \vv.
\ee

We note that, the sublaplacian $\mathcal{L}$ defined in (\ref{sublaplacian}) can be written in terms of the almost symplectic basis by
\be \label{Lsymplectic}
\mathcal L=\sum_{j=1}^{n}\left(X_j(\la)^2+Y_j(\la)^2\right).
\ee
This follows from the fact that the right hand sides of (\ref{sublaplacian}) and (\ref{Lsymplectic}) are both left invariant differential operators on $G$ which agree at the identity since the Euclidean Laplacian is invariant under an orthonormal basis change. It is well known (see \cite{R}) that
\bes
\pi_{\la}(\mathcal L)=-\sum_{j=1}^{n}\lt(-\f{\pl}{\pl\xi_j^2}+d_j^2(\la)\xi_j^2\rt)=-H(d(\la)),
\ees
where $d(\la)=(d_1(\la),d_2(\la),\ldots,d_n(\la))$.

We define a unitary operator $U(r)$ on $L^2(\R^n)$ as
\bes
U(r)\phi(\xi)=\prod_{j=1}^{n}r_j^{1/4} \phi(\sqrt{r_1}\xi_1,\ldots,\sqrt{r_n}\xi_n), \quad \txt{ for } \phi \in L^2(\R^n),
\ees
where $r=(r_1,\cdots,r_n) \in (0,\infty)^n$ and $\xi = (\xi_1, \cdots, \xi_n) \in \R^n$. We also define for each $\la \in \Lambda$
\bes
\phi_{\al}^{d(\la)}=U(d(\la))\phi_\al, \quad \txt{ for } \al \in \N^n,
\ees
where $\phi_{\al}$ denote the normalized Hermite functions on $\R^n.$ It is known that for any representation $\pi$ of $G$ realized on a Hilbert space $\mathcal{H}$ and for any element $A$ of the universal enveloping algebra of left invariant differential operators on $G$ we have
\bes
A(\pi(g)u,w)=(\pi(g)\pi(A)u,w), \quad \txt{ for } u,w \in H,~ g \in G.
\ees
If we plug in $\pi=\pi_{\la},$ $A=\mathcal{L},$ $u = w = \phi_{\al}^{d(\la)}$ and $w = \phi_{\beta}^{d(\la)}$ respectively, then
we get that
\bes
\mathcal{L}\left(\pi_{\la}(g)\phi_{\al}^{d(\la)},\phi_{\beta}^{d(\la)}\right)
=\left(\pi_{\la}(g)\pi_{\la}(\mathcal{L})\phi_{\al}^{d(\la)},\phi_{\beta}^{d(\la)}\right).
\ees
Since it is known that (see \cite{PT})
\bes
\pi_{\la}(\mathcal{L})\phi_{\al}^{d(\la)}=- (2\al+1)\cdot d(\la) ~ \phi_{\al}^{d(\la)},
\ees
so it follows that
\bes
\mathcal{L}\left(\pi_{\la}(g)\phi_{\al}^{d(\la)},\phi_{\beta}^{d(\la)}\right) =- (2\al+1)\cdot d(\la)\left(\pi_{\la}(g)\phi_{\al}^{d(\la)},\phi_{\beta}^{d(\la)}\right).
\ees
We define
\bes
\Phi_{\al \beta}^{d(\la)}(z)=\left(\pi_{\la}(z,0)\phi_{\al}^{d(\la)},\phi_{\beta}^{d(\la)}\right), \quad \txt{ for } z \in \vv.
\ees
Noting that
\bes
\left(\pi_{\la}(z,t)\phi_{\al}^{d(\la)},\phi_{\beta}^{d(\la)}\right) = \Phi_{\al \beta}^{d(\la)}(z) ~ e^{i \la(t)},
\ees
it follows that
\be \label{eigenfunction}
\mathcal{L}\left(\Phi_{\al \beta}^{d(\la)}(z) e^{i \la(t)}\right)=-(2\al+1)\cdot d(\la)\left(\Phi_{\al \beta}^{d(\la)}(z) e^{i \la(t)}\right), \quad \txt{ for } (z,t) \in G.
\ee

\section{Proof of Theorem \ref{theorem} for two step MW groups}

In this section, we are going to prove Theorem \ref{theorem} on two step MW groups. We will follow the method of proof in the case of Heisenberg groups in \cite{S}. However, in this case, the method is more complicated because here the symplectic basis depends on the the elements $\lambda \in \Lambda.$ In order to proceed with the proof, we wish to define an analogous distribution on $\R^k$ which turns out to be supported on the unit sphere. We will need few lemmas to be able to define this distribution. First, we will calculate the partial derivatives of $\Phi_{\al \alpha}^{d(\cdot)}$ in the $\la$-variable.

\begin{lem} \label{lemphi}
For any $l \in \{1, \cdots, k\},$ the partial derivative $\ds{\frac{\partial \Phi_{\alpha \alpha}^{d(\cdot)}}{\partial \la_l} }$ can be written as a finite linear combination of functions of the form
\be \label{form}
\tau_{\alpha}(\cdot) \Phi_{\alpha+\gamma ~ \alpha + \delta}^{d(\cdot)}
\ee
where $\tau_{\alpha}$ is a smooth function of $\la$ on $\Lambda$ depending on $\alpha$ and $\gamma, \delta \in \{ 0, \pm e_j: j = 1, \cdots, n  \}$. Here the total number of terms in the above linear combination is independent of $\alpha.$
\end{lem}

\begin{proof}
Realizing $\ds{z = \sum_{j=1}^n z_j(\la) Z_j(\la)}$ as $\left(z_1(\lambda), \cdots, z_n(\lambda)\right),$ it is easy to check that
$$\Phi_{\alpha \alpha}^{d(\la)}(z)= \Phi_{\alpha \alpha}\left(\sqrt{d_1(\la)}z_1(\lambda), \cdots, \sqrt{d_n(\la)}z_n(\lambda)\right),$$
where $\Phi_{\alpha \alpha}$ are the special Hermite functions on $\C^n$ (see \cite{T}). Therefore,
\beas
&& \frac{\partial}{\partial \la_l} \left(\Phi_{\alpha \alpha}^{d(\la)}(z)\right)\\
&=& \sum_{j=1}^n \frac{1}{2d_j(\la)} \frac{\partial d_j}{\partial \la_l}(\la) \left(z_j \partial_{z_j} + \bar{z_j} \partial_{\bar{z_j}}\right) (\Phi_{\alpha \alpha}) \left(\sqrt{d_1(\la)}z_1(\lambda), \cdots, \sqrt{d_n(\la)}z_n(\lambda)\right)+\\
&&  \sum_{j=1}^n \sqrt{d_j(\la)}\left(\frac{\partial z_j}{\partial\la_l}(\la)(\partial_{z_j}\Phi_{\alpha \alpha})+ \frac{\partial \bar{z_j}}{\partial\la_l}(\la)(\partial_{\bar{z}_j}\Phi_{\alpha \alpha})\right) \left(\sqrt{d_1(\la)}z_1(\lambda), \cdots, \sqrt{d_n(\la)}z_n(\lambda)\right)\\
&=&I_1+I_2.\\
\eeas

Now, using the relations (1.3.19), (1.3.20), (1.3.23) and (1.3.24) of \cite{T} we get that
\bes
\left(z_j\partial_{z_j}+\bar{z}_j\partial \bar{z}_j\right) \Phi_{\alpha \alpha}= (\alpha_j+1) \Phi_{\alpha+e_j ~ \alpha+e_j}- \alpha_j \Phi_{\alpha-e_j ~ \alpha-e_j} - \Phi_{\alpha \alpha}.
\ees
Since $d_j$ is a non-zero smooth function on $\Lambda$ for each $j,$ it follows that  we can write $I_1$ as a finite linear combination of terms which are of the form (\ref{form}).

In order to study $I_2,$ we first note that if $\ds{z = \sum_{j=1}^n z_i Z_i,}$ then from the definition of the transformation $D_\la$ discussed earlier, we get that
$$\left( \begin{array}{cc} z_1(\la) \\ z_2(\la)\\.\\.\\.\\z_n(\la) \end{array} \right)=D_\la \left( \begin{array}{cc}z_1 \\ z_2\\.\\.\\.\\z_n  \end{array} \right).$$
If $D_\la = \left(a_{pq}(\la)\right)_{p,q=1}^n$ and $D_\la^{-1} = \left(b_{pq}(\la)\right)_{p,q=1}^n,$
then we have
$$\frac{\partial z_j}{\partial \la_l}(\la)=\sum_{\nu=1}^n \frac{\partial a_{j\nu}}{\partial \la_l}(\la)~z_\nu=\sum_{\nu=1}^n\sum_{\rho=1}^n \frac{\partial a_{j\nu}}{\partial \la_l}(\la) ~ b_{\nu \rho}(\la)z_{\rho}(\la).$$
Hence, it follows that
\beas
&& \frac{\partial z_j}{\partial\la_l}(\la)(\partial_{z_j}\Phi_{\alpha \alpha}) \left(\sqrt{d_1(\la)}z_1(\lambda), \cdots, \sqrt{d_n(\la)}z_n(\lambda)\right) \\
&=& \sum_{\nu=1}^n\sum_{\rho=1}^n \frac{1}{\sqrt{d_\rho(\nu)}}\frac{\partial a_{j\nu}}{\partial \la_l}(\la)b_{\nu \rho}(\la)(z_{\rho}\partial_{z_j})(\Phi_{\alpha \alpha}) \left(\sqrt{d_1(\la)}z_1(\lambda), \cdots, \sqrt{d_n(\la)}z_n(\lambda)\right).
\eeas
We can evaluate $\ds{\frac{\partial \bar{z_j}}{\partial \la_l}(\la)(\partial_{\bar{z}_j}\Phi_{\alpha \alpha})}$ similarly.
Since $a_{pq}$ and $b_{pq}$ are smooth functions on $\Lambda$ for $p,q=1, \cdots, n,$ using the relations (1.3.19), (1.3.20), (1.3.23) and (1.3.24) of \cite{T} again, we get that $I_2$ can also be expressed as a finite linear combination of terms  of the form (\ref{form}). Hence the lemma is proved.
\end{proof}

The next lemma will give an eigenfunction of $\mathcal{L}$ with eigenvalue $-|\la|$ for $\la \in \Lambda.$

\begin{lem}
For $\la \in \Lambda$, there exists a function $h_{\la}$ on $G$ such that
\bea \label{eigen}
\mathcal{L} h_{\la} = - |\la| h_{\la}, \quad \txt{ for }\ \la \in \Lambda.
\eea
\end{lem}

\begin{proof}
We note that if $\la = |\la| \omega$ is the polar representation of $\la \in \R^k$ where $\omega \in S^{k-1},$ one can easily see that
\be \label{relation}
B_\la = |\la| B_{\omega}, \quad X_j(\la) = X_j(\omega), \quad Y_j(\la) =  Y_j(\omega) \quad \txt{ and } \quad d_j(\la) = |\la| d_j(\omega),
\ee
for each $j= 1, 2, \cdots, n.$ We now define
\bes
\tilde{\la} = \frac{|\la|}{(2\alpha+1)\cdot d(\la)} \la.
\ees
If $\tilde{\la} = |\tilde{\la}| \tilde{\omega},$ then from the relation (\ref{relation}) we get that $d(\omega) = d(\tilde{\omega}).$ Then it also follows that
\be \label{lambda}
\la = \frac{(2\alpha+1)\cdot d(\tilde{\la})}{|\tilde{\la}|} \tilde{\la}.
\ee
We now define
\bes
h_{\la}(z,t) = \Phi_{\al \alpha}^{d(\tilde{\la})}(z) e^{i \tilde{\la}\cdot t}.
\ees
Observing that $(2\al+1)\cdot d(\tilde{\la}) =  |\la|$ and using (\ref{eigenfunction}) we obtain
\bes
\mathcal{L} h_{\la}(z,t) = -|\la| ~ h_{\la}(z,t), \quad \txt{ for } (z,t) \in G.
\ees
\end{proof}

Now we are in a position to define the distribution on $\R^k$ mentioned in the beginning of this section. Let us fix a function $\varphi \in C_c^{\infty}(\Lambda).$ For a test function $\psi \in C_c^{\infty}(\R^k),$ we define
\be \label{distribution} T^{\varphi}(\psi)=\int_G f_0(z,t) \int_\Lambda h_\la(z,t) \psi(\la) \varphi(\lambda) d\la ~ dz ~ dt, \ee
which resembles a kind of Fourier transform of $f_0.$ We note that in order to define the distribution $T^{\varphi}$ on $\R^k,$ we have effectively considered the test function $\varphi \psi$ supported on $\Lambda$ instead of $\R^k$ to avoid some messy boundary terms while doing few integration by parts arguments later in the proof. In the next lemma, we prove that $T^{\varphi}$ is a distribution on $\R^k$.
\begin{lem} \label{distlemma}
$T^{\varphi}$ is a distribution on $\R^k$ with $|T^{\varphi}(\psi)| \lesssim \|\psi\|_{2},$ for $\psi \in C_c^{\infty}(\R^k),$ where $\|\psi\|_{l} = \sup\{|D^{\rho}\psi(\la)| : \la \in \R^k, |\rho| \leq l\}.$
\end{lem}
\begin{proof}
For  $\psi \in C_c^{\infty}(\R^k),$ we note that
\beas
|T^{\varphi}(\psi)|
&=&\left|\int_G f_0(z,t) \int_\Lambda e^{i\tilde{\la}\cdot t}\Phi_{\alpha \alpha}^{d(\tilde{\la})}(z) \varphi(\la) \psi(\la) d\la dz dt\right| \\
&=&\left|\int_G f_0(z,t) \frac{1}{1+|t|^2}\int_\Lambda e^{i\tilde{\la}\cdot t} \left(1-\sum_{l=1}^k \frac{\partial^2}{\partial \tilde{\la}_l^2}\right) \left(\Phi_{\alpha \alpha}^{d(\tilde{\la})}(z) \varphi(\la) \psi(\la) \right) d\la dz dt\right| \\
&\leq& I_1 + \sum_{l=1}^k I_{1l} + \sum_{l=1}^k I_{2l} + 2\sum_{l=1}^k I_{3l},
\eeas
where
\beas
I_1 &=&\left|\int_G \frac{f_0(z,t)}{1+|t|^2}\int_\Lambda e^{i\tilde{\la}\cdot t}\Phi_{\alpha \alpha}^{d(\tilde{\la})}(z) \varphi(\la) \psi(\la) d\la dz dt\right|, \\
I_{1l} &=& \left|\int_G  \frac{f_0(z,t)}{1+|t|^2}\int_\Lambda e^{i\tilde{\la}\cdot t} \frac{\partial^2}{\partial \tilde{\la}_l^2} \left(\Phi_{\alpha \alpha}^{d(\tilde{\la})}(z)\right) \varphi(\la) \psi(\la) d\la dz dt\right|, \\
I_{2l} &=& \left|\int_G  \frac{f_0(z,t)}{1+|t|^2}\int_\Lambda e^{i\tilde{\la}\cdot t} \Phi_{\alpha \alpha}^{d(\tilde{\la})}(z) \frac{\partial^2}{\partial \tilde{\la}_l^2} \left( \varphi(\la) \psi(\la)\right) d\la dz dt\right|, \\
I_{3l} &=& \left|\int_G  \frac{f_0(z,t)}{1+|t|^2}\int_\Lambda e^{i\tilde{\la}\cdot t} \frac{\partial}{\partial \tilde{\la}_l} \left(\Phi_{\alpha \alpha}^{d(\tilde{\la})}(z)\right) \frac{\partial}{\partial \tilde{\la}_l} \left(\varphi(\la) \psi(\la)\right) d\la dz dt\right|. \\
\eeas
for  $l = 1, 2, \cdots, k.$ Using Lemma \ref{lemphi} and (\ref{lambda}) we can write each of $I_1$, $I_{1l}$, $I_{2l}$ and $I_{4l}$ as finite linear combinations of the form
\bes
\left|\int_G  \frac{f_0(z,t)}{1+|t|^2}\int_\Lambda e^{i\tilde{\la}\cdot t} ~  \Theta_{\alpha}(\la) ~\Phi_{\alpha + \gamma ~ \alpha + \delta}^{d(\tilde{\la})}(z) ~ \tilde{\varphi}(\la) ~ \tilde{\psi}(\la) ~ d\la ~ dz ~ dt\right|,
\ees
where $\Theta_{\alpha}$ is a smooth function of $\la$ on $\Lambda$ depending on $\alpha$ and the quantities $\gamma, \delta \in \{ 0, \pm e_j, \pm 2 e_j: j = 1, \cdots, n  \}$ and $\tilde{\varphi}, \tilde{\psi}$ are some derivatives of $\varphi, \psi$ of orders $\leq 2$ respectively. Since $\tilde{\varphi} \in C_c^{\infty}(\Lambda)$ and $\tilde{\psi} \in C_c^{\infty}(\R^k),$ we see that the above expression can be bounded by
\beas
&& C^{\varphi}_{\alpha} M \left( \int_\mathfrak{z}  \frac{dt}{1+|t|^2} \right) \int_{\mathfrak{v}} \int_\Lambda \left|\Phi_{\alpha + \gamma ~ \alpha + \delta}^{d(\tilde{\la})}(z) ~ \tilde{\psi}(\la) \right|~ d\la ~ dz \\
&\lesssim& \left( \int_\mathfrak{v} \left|\Phi_{\alpha + \gamma ~ \alpha + \delta}(z)\right| dz \right) \int_{\Lambda} \frac{|\tilde{\psi}(\la)|}{d_1(\la)^{1/2} \cdots d_n(\la)^{1/2}} d\la \\
&\lesssim& \|\psi\|_{2}.
\eeas
Hence we can conclude that $T^{\varphi}$ is a distribution.
\end{proof}

Now we are in a position to prove the main theorem of this paper (Theorem \ref{theorem}) for the case of two step MW groups.

\begin{center}
\underline{\textbf{Proof of Theorem \ref{theorem} for two step MW groups}}
\end{center}

First we wish to show that the distribution $T^{\varphi}$ defined in (\ref{distribution}) is supported on the unit sphere in $\R^k.$ We consider a test function $\psi$ supported in the set $\{\la \in \R^k : |\lambda| > 1 \} .$ Then using the conditions (\ref{seq}) and (\ref{bdd}) we can easily see that for any $l \in \N$
\bes
|T^{\varphi}(\psi)| = \left| \left\langle f_0,\int_{\Lambda } h_{\lambda} \varphi(\la) \psi(\la) d\lambda \right\rangle_G \right| =\left| \left\langle \mathcal{L}^lf_{0},\int_{\Lambda } h_{\lambda}\frac{\varphi(\la) \psi(\la)}{(-|\lambda|)^{l}} d\lambda \right\rangle_G \right|\leq  M
\left\|\int_{\Lambda} h_{\lambda}
\frac{\varphi(\la) \psi(\la)}{(-|\lambda|)^{l}} d\lambda \right\|_{L^1(G)}.
\ees
Applying the arguments used in Lemma \ref{distlemma} we get that there exists $K>1$ such that
\beas\left\|\int_{\Lambda} h_{\lambda}
\frac{\varphi(\la) \psi(\la)}{(-|\lambda|)^{l}} d\lambda \right\|_{L^1(G)}\lesssim \frac{\|\psi\|_{2}}{K^l},
\eeas
which goes to $0$ as $l \rightarrow \infty.$ So it follows that $T^{\varphi}$ is supported on $\{ \la \in \R^k : |\la| \leq 1 \}.$ Similarly considering a test function $\psi$ supported in $\{ \la \in \R^k : |\la| < 1\},$ we can show that $T^{\varphi}$ is supported on $\{ \la \in \R^k: |\la| \geq 1 \}.$ So we can conclude that $T^{\varphi}$ is supported on the set $\{ \la\in \R^k : |\la| = 1 \}.$

Using structure theorem for distributions supported on the unit sphere, we get that
\bes
\ds{ \left\langle f_0,\int_{\Lambda } h_{\lambda} \varphi(\la) \psi(\lambda) d\lambda \right\rangle_G} = \ds{\sum_{j=0}^N \left. \left(r\frac{\partial}{\partial r} \right)^j {\left\langle T_j,\psi_{r}
\right\rangle}_{S^{k-1}} \right|_{r=1},}
\ees
where $T_j$ are distributions on the unit sphere $S^{k-1}$ and $\psi_r(\omega) = \psi(r \omega)$ where $r>0, \omega \in S^{k-1}.$ Now, using the relation (\ref{seq}) we get that for any $l \in \N,$
\bes
\ds{ \left\langle f_l,\int_{\Lambda } h_{\lambda} \varphi(\la) \psi(\lambda) d\lambda \right\rangle_G} = \ds{\sum_{j=0}^N \left. \left(r\frac{\partial}{\partial r} \right)^j (-r)^{l}{\left\langle T_j,\psi_{r}
\right\rangle}_{S^{n-1}} \right|_{r=1}.}
\ees
Using (\ref{bdd}) we note that the quantity on the left hand side is uniformly bounded in $l$ whereas the quantity on the right hand side is a polynomial of degree $N$ in $l.$ So we can infer that $N=0$ in the above expansion and thereby obtain that
\bes
\left\langle f_0 + f_1,
\int_{\Lambda} h_{\lambda} \varphi(\la) \psi(\lambda) d\lambda \right\rangle_G = 0.
\ees
We note that the above is true for each $\varphi \in C_c^{\infty}(\Lambda)$ and for each $\psi \in C_c^{\infty}(\R^k).$ Using Urysohn's Lemma given $\varphi \in C_c^{\infty}(\Lambda)$ we can choose $\psi \in C_c^{\infty}(\R^k)$ such that $\varphi \psi \equiv \varphi$ on $\Lambda.$ So we get that
\bes
\left\langle f_0 + f_1,
\int_{\Lambda} h_{\lambda} \varphi(\la) d\lambda \right\rangle_G = 0, \quad \txt{ for each } \varphi \in C_c^{\infty}(\Lambda).
\ees
However, we need to show that $f_0 + f_1 = 0.$ For this, we define $f=f_0 + f_1$ and $\ds{F_{\alpha} = \int_{\Lambda} h_{\lambda} \varphi(\lambda) d\lambda}.$

We consider the class of functions
\bes
\mathcal{F} = \left\{ g \in L^1(G): \int_G fg = 0 \right\}.
\ees
Clearly, $F_{\alpha} \in \mathcal{F}$ for each $\alpha \in \mathbb{N}^n.$ Since $\mathcal L$ is left invariant, the assumption (\ref{seq}) of the theorem holds for all left translates of $f_k$'s. It follows that all left translates of $F_{\alpha}$ are in $\mathcal{F}.$ Moreover, if we define $\tilde{g}(z,t) = g((z,t)^{-1}) = g(-z, -t)$ for $(z,t) \in G,$ then using the relation $\phi_{\alpha \alpha}^{d(\la)}(z) = \phi_{\alpha \alpha}^{d(\la)}(-z)$ it is easy to see that $\tilde{F_\alpha} \in \mathcal{F}$ for each $\alpha \in \N^n.$ Further we can deduce that the right translates of $F_{\alpha}$ are in $\mathcal{F}.$ Hence it follows that the two-sided ideal generated by the set $\{F_\alpha : \alpha \in \N^n\}$ is contained inside $\mathcal{F}.$

Now, we note that for each $\la\in\Lambda$, the group Fourier transform of $F_\alpha$ is given by
\bes\widehat{F_\alpha}(\la) = (2\pi)^n\left(\prod_{j=1}^n d_j (\la)\right)\varphi(\la)P^{d(\la)}_\alpha\neq 0,\ees
 where $P^{d(\la)}_\alpha$ is the projection on the eigenspace associated to the eigenvalue $(2\alpha+1)\cdot d(\la)$ corresponding to the operator $H(d(\la)).$ One can also easily check that the Fourier transforms associated to the one dimensional irreducible representations of  $F_0$ vanish nowhere for suitable $\varphi$.
Now we can apply the Wiener-Tauberian type theorem corresponding to the two-sided action of $G$ on itself proved in \cite{L} to get that $\mathcal{F}=L^1(G)$. Hence, we can conclude that $f=0$.
\qed

\section{Proof of Theorem \ref{theorem} for two step non-MW groups}
Let $G$ be a non-MW group and $\g$ be the corresponding Lie algebra with the decomposition $\g = \vv \oplus \z$. Following \cite{MR}, we can associate a two step MW Lie algebra $\h$ with $\g$ by the following
\bes
\h = (\vv \times \vv^*) \oplus (\z \times \R) = \h_\vv \oplus \h_\z
\ees
with the Lie bracket $[\cdot, \cdot] : \h_\vv \times \h_\vv \ra \h_\z$ given by
\bes
[(V, \eta), (V', \eta')] = ([V,V'], \eta'(V) - \eta(V')).
\ees
Therefore, the group $H$ associated with the Lie algebra $\h$ can be identified with $\R^m\times \R^m\times \R^k\times \R^k$ endowed with the group operation
\bes
(v, \eta, z, t)(v', \eta', z', t')=\left(v+v',\eta+\eta', z+ z'+\frac{1}{2}[v,v'], t+t'+\frac{1}{2}(\eta'(v) - \eta(v'))\right).
\ees
Let $V_1, V_2,\cdots, V_m$ be left invariant vector fields which form a basis of $\vv$ and we coordinatise $\vv$ with respect to this basis by $V=\sum_{j=1}^m v_j V_j$. Then the following left invariant vector fields will form a basis of $\h_\vv$:
\beas
\widetilde{V_i}&=&V_i+\frac{1}{2}\eta_i \frac{\partial}{\partial t},\quad 1\leq i\leq m,\\
\widetilde{V}_{m+i}&=& \frac{\partial}{\partial \eta_i}-\frac{1}{2}v_i\frac{\partial}{\partial t},\quad 1\leq i\leq m.
\eeas
Hence the sub-Laplacian of $\h$ is given by
$$\widetilde{\mathcal{L}}=\sum_{i=1}^{2m} \widetilde{V_i}^2.$$

\begin{center}
\underline{\textbf{Proof of Theorem \ref{theorem} for two step non-MW groups}}
\end{center}
Let us define $\tilde{f}_l(v, \eta,z,t)= f_l(v, z)$.
Then it can be easily seen that
\beas
\widetilde{V_i}\tilde{f}_l&=&V_i f_l,\quad 1\leq i\leq m,\\
\widetilde{V}_{m+i}\tilde{f}_l&=& 0,\quad 1\leq i\leq m.
\eeas
Hence, we have
\bes
\widetilde{\mathcal{L}} \tilde{f}_l(v,\eta,z,t)=\mathcal{L}f_l(v,z)=f_{l+1}(v,z)=\tilde{f}_{l+1}(v, \eta,z,t), \quad l\in \Z.
\ees
As, $\h$ is MW, we can apply Theorem \ref{theorem} in case of MW groups to get that
\bes
\widetilde{\mathcal{L}} \tilde{f}_0=-  \tilde{f}_0.
\ees
So, finally we have
\bes
\mathcal{L}f_0=-  f_0.
\ees
This completes the proof of Theorem \ref{theorem}.
\qed

\begin{center}
\bf{Acknowledgements}
\end{center}

This work was supported by Department of Science and Technology, India (INSPIRE Faculty Award to Sayan Bagchi). We would like to thank Prof. Swagato K. Ray for suggesting this problem and also Prof. S. Thangavelu and Prof. E. K. Narayanan for some useful discussions and suggestions.

\end{document}